\documentclass[11pt]{article} 

\usepackage{amssymb,latexsym,amsmath}
\usepackage{epsfig}
\usepackage{caption}
 
 \usepackage{amsmath}
\usepackage{amsfonts}
\usepackage{array}
\usepackage{dsfont}
\usepackage{amssymb}
\usepackage{amsthm}

\newtheorem{assumption}{Assumption}

 \newtheorem{remark}{Remark}
\newtheorem{theorem}{Theorem}
\newtheorem{proposition}{Proposition}
\newtheorem{corollary}{Corollary}

\newtheorem{lemma}{Lemma}

\usepackage[colorlinks=true,breaklinks=true,bookmarks=true,urlcolor=blue,
     citecolor=blue,linkcolor=blue,bookmarksopen=false,draft=false]{hyperref}

\def\argmin{\mathop{\rm arg\, min}}

\def\B{{\mathcal B}}

\def\P{{\mathcal P}}

\newcommand{\Zplus}{\mathbb{Z}_+}

\usepackage[utf8]{inputenc}

\begin{document}

\sloppy
\title{Approximate Q-Learning for Controlled Diffusion Processes and its Near Optimality\thanks{
E. Bayraktar is partially supported by the National Science Foundation under grant DMS-2106556 and by the Susan M. Smith chair.}
}
\author{Erhan Bayraktar and Ali Devran Kara
\thanks{The authors are with the Department of Mathematics,
     University of Michigan, Ann Arbor, MI, USA,
     Email: \{erhan,alikara\}@umich.edu}
     }
\maketitle

\maketitle
\begin{abstract}
We study a Q learning algorithm for continuous time stochastic control problems. The proposed algorithm uses the sampled state process by discretizing the state and control action spaces under piece-wise constant control processes. We show that the algorithm converges to the optimality equation of a finite Markov decision process (MDP). Using this MDP model, we provide an upper bound for the approximation error for the optimal value function of the continuous time control problem. Furthermore, we present provable upper-bounds for the performance loss of the learned control process compared to the optimal admissible control process of the original problem. The provided error upper-bounds are functions of the time and space discretization parameters, and they reveal the effect of different levels of the approximation: (i) approximation of the continuous time control problem by an MDP, (ii) use of piece-wise constant control processes, (iii) space discretization. Finally, we state a time complexity bound for the proposed algorithm as a function of the time and space discretization parameters.
\end{abstract}

\tableofcontents
\section{Introduction}
The goal of this paper is to develop a convergent learning algorithm for controlled diffusion processes when the decision maker only has access to the state process and the cost realizations, and establish error bounds for the performance of the learned control process compared to the optimal admissible control as a function of the algorithm parameters. 
\subsection{Preliminaries}
We start with the setup of the paper by defining the dynamics of the control problem.
The objective is to study a controlled diffusion process, $X(\cdot)$, given by the following stochastic differential equation
\begin{align}\label{diff}
X(t)=X_0+\int_0^{t}b(X(s),u(s))ds+\int_0^t\sigma(X(s),u(s))dB(s)
\end{align}
for $t \geq 0$ where $X(t)\in \mathds{X}$, and $\mathds{X}\subset \mathds{R}$. $B(\cdot)$ is the driving noise which is assumed to be a Wiener process, and  $u(\cdot)\in\mathds{U}$ is the control process with measurable paths. We assume that the control is non-anticipative such that for $0\leq s <t$, the noise increments $B(t)-B(s)$, are independent of $B(y),u(y)$ for $y\leq s$. We note that the results presented in this paper will be applicable for multidimensional spaces as well, however, we will assume that everything is one-dimensional for ease of notation.

We will later put assumptions on the model, which will guarantee the existence of strong solutions to the diffusion process (\ref{diff}), under admissible control processes (Assumption \ref{diff_assmp}).

The objective of the controller is to minimize the following infinite horizon discounted cost function
\begin{align}\label{cost}
W_\beta(x_0,u)=E\left[\int_0^\infty e^{-\beta s}c\left(X(s),u(s)\right)ds\right]
\end{align}
where the expectation is with respect to the initial point $X(0)=x_0$ and the given control process $u$, the stage-wise cost function is denoted by $c:\mathds{X}\times \mathds{U}\to \mathds{R}$, and $\beta$ is some discount factor. 

The optimal cost is denoted by
\begin{align*}
W_\beta^*(x_0):=\inf_u W_\beta(x_0,u)
\end{align*}
where the infimum is taken over all admissible control processes.

\subsection{Problem Formulation and a Proposed Algorithm}
Our goal is to provide a reinforcement learning algorithm which converges to an approximately optimal control policy under certain assumptions. In particular, we will show that if the standard Q learning algorithm (\cite{Watkins,TsitsiklisQLearning,jaakkola1994convergence}) is used for the state process after time and space discretization, we can find sufficient conditions for the algorithm to converge. Furthermore, we will provide provable error bounds for the performance of the policies learned through these iterations. We note that the standard Q learning algorithm is not readily applicable for the diffusion model we consider in the paper. The classical Q learning algorithm is designed for finite Markov chains, where the Markov property of the data is crucially used to prove the convergence of the algorithm. In the model we consider here, time discretization breaks the Markov property of the induced chain unless the control policies are selected carefully, namely selected as piece-wise constant functions. Furthermore, even when the sampled process which lives on the original continuous state space is Markov, the process constructed using space discretization will not be a Markov chain. Hence, one needs to alter the convergence proof significantly when the space is discretized, in order to show the classical Q learning algorithm converges under time and space discretization.

We assume that the drift, diffusion, and the cost functions $b,\sigma$ and $c$ are not known. By observing the state process $X(t)$ and the cost realizations $c(X(t),u(t))$, we try to learn the Q values of a finite Markov decision process (MDP), which will be shown to well approximate the original problem. Furthermore, we keep track of the state and the cost processes via sampling the time and discretizing $\mathds{X}$ and $\mathds{U}$ spaces. To this end, we fix a sampling interval $h>0$, and finite subsets $\mathds{X}_h\subset \mathds{X}$,  $\mathds{U}_h\subset \mathds{U}$. We put the $h$ dependence on the finite spaces to recover the cases where the space discretization rate depends on the time discretization rate. We also define the mapping  $\phi_\mathds{X}:\mathds{X}\to\mathds{X}_h$, to map the original value of the state variables to the discrete set $\mathds{X}_h$, e.g. a natural choice would be a nearest neighbour map. 

For the exploration phase, we use a piecewise constant control process, such that at $i\times h$, for $i=0,1,\dots$, some $\hat{u}\in\mathds{U}_h$ is chosen according to the exploration policy and is applied for the time interval $[i\times h,(i+1)\times h)$. Hence, the resulting exploration process $\hat{u}(t)$ is a piecewise constant process. 

After we fix the exploration control process $\hat{u}(t)$, we observe the controlled Markov chain process $X_n$ 
\begin{align*}
X_n:=X(n\times h),\quad \text{ for } n=0,1,2\dots.
\end{align*}
Note that $X_n$ takes values from a continuous set $\mathds{X}$. We further map these values to the finite set $\mathds{X}_h$, to construct the stochastic process $\hat{X}_n$, which is defined as
\begin{align}\label{chain}
\hat{X}_n:=\phi_{\mathds{X}}\left(X_n\right)
\end{align}
The process $\hat{X}_n$ is the discretized version of the sampled diffusion process, and hence it takes values from $\mathds{X}_h$. We note that the learning algorithm will be constructed using the process $\{\hat{X}_n\}_{n}$.

We now present the proposed algorithm formally as follows:
\begin{enumerate}
\item Chose a sampling interval $h>0$.
\item Choose finite subsets $\mathds{X}_h\subset \mathds{X}$ and $\mathds{U}_h\subset\mathds{U}$.
\item Define a mapping $\phi_\mathds{X}:\mathds{X}\to\mathds{X}_h$ (e.g. a nearest neighbour map).
\item Select a $\mathds{U}_h$ valued piecewise constant exploration process $\hat{u}(t)$.
\item Observe the discrete process $\hat{X}_n$ (see (\ref{chain})), and for all $(\hat{x},\hat{u})\in\mathds{X}_h\times\mathds{U}_h$, initiate the Q values by choosing $Q_0(\hat{x},\hat{u})$, choose the learning rates $\alpha_k(\hat{x},\hat{u})$ (see Assumption \ref{learning_assmp}), using the cost realization $c(X(n\times h),\hat{u})$ update
\begin{align}\label{QPOMDP}
Q_{k+1}(\hat{x},\hat{u})=(1-\alpha_k&(\hat{x},\hat{u}))Q_k(\hat{x},\hat{u})\nonumber\\
+&\alpha_k(\hat{x},\hat{u})\left(c(X(k\times h),\hat{u})\times h+\beta_h \min_{v\in\mathds{U}_h} Q_k\Big(\hat{X}_{k+1},v\Big)\right)
\end{align}
where $\beta_h=e^{-\beta\times h}$, and $\hat{X}_{k+1}$ is the sampled state we observe following $\hat{X}_k=\hat{x}$.


\item In Section \ref{q_conv}, we will prove that under suitable assumptions, the iterations $Q_k:\mathds{X}_h\times\mathds{U}_h\to \mathds{R}$, converge almost surely to some $Q^*:\mathds{X}_h\times\mathds{U}_h\to \mathds{R}$. In Section \ref{MC_section}, we will show that limit values are the $Q$-values of some finite controlled Markov chain. Using $Q^*$, we define a mapping $\gamma_h:\mathds{X}_h\to \mathds{U}_h$ such that for any $\hat{x}\in\mathds{X}_h$
\begin{align*}
\gamma_h(\hat{x})=\argmin_{\hat{u}\in\mathds{U}_h}Q^*(\hat{x},\hat{u}).
\end{align*}

\item Define the control process $u_h(t)$ such that
\begin{align*}
u_h(t)=\gamma_h\left(\phi_\mathds{X}(X(i\times h)\right), \text{ for } t\in [i\times h, (i+1)\times h),
\end{align*}
that is $u_h$ is a piece-wise constant process, which changes value at the sampling instances according to the learned map $\gamma_h$.
\end{enumerate}
We are then interested in the following question:

\textbf{Problem:} If we use the learned control $u_h$ for the continuous time model, what do we lose compared the optimal admissible control process. In other words, we are interested in the difference
\begin{align}\label{error}
W_\beta(x_0,u_h)-W_\beta^*(x_0).
\end{align}
We will try to bound this term in terms of the sampling interval $h$ and the state and control action spaces' discretization errors $L_{\mathds{X}}$ and $L_{\mathds{U}}$ that are defined as 
\begin{align*}
L_{\mathds{X}}:=&\sup_{x\in\mathds{X}}\left|x-\phi_\mathds{X}(x)\right|,\\
L_{\mathds{U}}:=&\sup_{u\in\mathds{U}}\min_{u_h\in\mathds{U}_h}\left|u-u_h\right|.
\end{align*}

We will also analyze the asymptotic case, that is, we will show that under certain assumptions 
\begin{align*}
W_\beta(x_0,u_h)-W_\beta^*(x_0)\to 0
\end{align*}
as $h, L_{\mathds{X}},L_{\mathds{U}}\to 0$.

To answer these questions, we will show that the iterations in (\ref{QPOMDP}) converge to the optimal Q values of some finite Markov chain which approximates the value function of the diffusion process in (\ref{diff}) with the cost function (\ref{cost}).

\subsection{Literature Review}
The results and the techniques used in the paper are related to reinforcement learning algorithms for continuous and discrete time control problems, finite time-space approximations of diffusion processes, and approximation methods for Markov decision processes. In the following, we summarize the related works in the literature by their main objectives:

\textbf{Reinforcement learning for discrete-time stochastic control problems} Optimal control of stochastic ( or deterministic) dynamical systems typically requires a perfect knowledge about the system components, however, the correct underlying model of the control problem is usually unknown or too complicated to work with. The objective of reinforcement learning is to estimate or learn the relevant information about the problem such as the value function or the optimal policy by interacting and observing the system. Majority of the literature, however, deals with discrete-time and sequential decision making problems. A popular reinforcement learning algorithm, called Q learning (\cite{Watkins}), for example, is proposed to learn the Q factors, which is closely related to the value function, for discrete time and space problems. Q learning is guaranteed to converge under mild assumptions for discrete settings without any information on the transition models or the cost function of the problem by observing the state and cost realizations (\cite{TsitsiklisQLearning,jaakkola1994convergence}). 

Even though Q learning is simple to implement and requires minimal knowledge about the system, it is not effective for large or continuous state and action spaces. To overcome the dimension challenges, one can try to learn an approximate version of the problem. The approximation can be done through various methods:  One can use function approximations for the optimal value function (see \cite{CsabaAlgorithms,tsitsiklis1997analysis}). To approximate the value functions, neural networks, state aggregation, or linear approximation techniques with finitely many linearly independent basis functions can be used. For state aggregation and linear approximation methods, convergence can be shown, however,  a rigorous error analysis is  usually not available. Some related work includes \cite{singh1995reinforcement,melo2008analysis,gaskett1999q,CsabaSmart,NNQlearning} and references therein. For the aforementioned works, typically, either a careful parametrization of the value functions or strong density assumptions on the transition kernels are required. However, for general continuous time problems, these assumptions might be too strong. For example, the sampled continuous time stochastic controlled process we will study in this paper, can only be shown to have weakly continuous transition dynamics (see Lemma \ref{kernel_x} and Lemma \ref{kernel_u}). In \cite{kara2021qlearning}, it is shown that, by choosing a finite subset of the action space and discretizing the state space, one can consistently learn nearly optimal control policies for systems with weakly continuous transition models.

\textbf{Approximations for continuous time control problems} Space approximation methods allow one to use learning algorithms for control problems with large state spaces, however, these methods only work for discrete time decision making problems. Control problems in continuous time, in general, are not feasible to work with numerically. Due to complex nature of the problems, solutions of the control process may not be available, which results in approximation attempts on the time domain, either through the state process or the optimality equation. \cite{kushner2001numerical} provides a general Markov chain approximation method for the controlled diffusion processes, by directly approximating the continuous time state process with a finite space controlled Markov chain. It is shown that under certain 'consistency' conditions, state process and value function approximations are asymptotically close to the solutions of the continuous time process. \cite{Krylov97,Krylov00} establish convergence rates for such approximation methods. Improved convergence rates for finite approximation methods are later presented in \cite{Jacobsen02,Jacobsen05}. We note that, these works study the approximation of an available model and do not focus on learning an approximate model when the dynamics are not available to the decision maker.

\textbf{Reinforcement learning in continuous time} Learning and planing in continuous time problems become much more challenging in continuous time mainly because of the complex dynamics and optimality equation of the problems.

 We first note that our main objective and contribution is to rigorously prove convergence of the discretization based model-free learning algorithms for general diffusion processes and to provide provable error bounds that clearly indicate the effect of space and time discretization.

A large number of papers dealing with learning of continuous time control problems (learning of value functions or control policies), considers linear dynamics and cost setting, and develop algorithms and theoretical results using the structural properties of this limited setting, see e.g. \cite{palanisamy2014continuous,bian2016value,vamvoudakis2017q,rajagopal2016neural,kontoudis2019kinodynamic,wang2020reinforcement,wang2020continuous} for some of the related papers that work with linear models. In our paper, we consider general non-linear controlled diffusion dynamics, where the only assumption we put on the dynamics is a standard continuity assumption, which is  required for existence results.

Another large set of studies in  learning of continuous time control problems consider deterministic dynamics see e.g. \cite{tassa2007least,yang2017hamiltonian,lutter2020hjb,kontoudis2019kinodynamic,lee2021policy,yildiz2021continuous,munos2006policy}. We note that the approximation, convergence and structural analysis are considerably more challenging for continuous time {\it{stochastic}} control problems. We further note that some of these studies focus on approximate solutions to the optimal control problem with an available model, instead of considering model-free learning methods.

Another related direction includes estimation and learning of dynamics for stochastic differential equations see e.g. \cite{yildiz2018learning, batz2018approximate, pereira2010learning, liao2019learning}. In these papers, the focus is on the learning of models for control-free stochastic processes governed by stochastic differential equations (SDEs), which differs from our objective of learning optimal value functions and-or near optimal policies in controlled stochastic processes.

Learning methods are also used to find (approximate) solutions to partial differential equations and to solve HJB equations (approximately) see e.g. \cite{tassa2007least,abu2005nearly, germain2021neural, lefebvre2022differential}. Note that these works approach to the solution problem with a model in hand ( i.e. model-based not model-free), where the focus is on finding the solution of the PDEs and HJB equations which is by itself a challenging problem.

The closest papers to ours are \cite{tallec2019making,baird1994reinforcement,munos1997reinforcement}. \cite{munos1997reinforcement} proposes a learning method using time and space discretizations for a general controlled diffusion process, and proves that the learned approximations are asymptotically optimal as the discretization parameters converge to 0. In our paper, we present error bounds in terms of the space and time discretization rate, which explicitly shows the effects of discretization. \cite{tallec2019making,baird1994reinforcement} proposes `advantage functions' for continuous time learning algorithms building on the observation that Q values are no longer informative for continuous time settings, and this observation is also widely used in the literature. In our results, we show that this claim should be approached with care. We show that as long as the space discretizaton rate is at least as high as the time discretization rate, Q values of the approximate MDP model well approximate the value function of the controlled diffusion problem.

\textbf{Our contributions} In this paper, we will study an approximate Q learning algorithm for general continuous time stochastic control problems by discretizing the time, and state and action spaces. The discretization in time will assumed to be uniform, however, state and action spaces can be discretized in a non-uniform way. We will show that the algorithm will converge under mild assumptions, even when the space quantization is non-uniform. Different from earlier works, we will then provide error bounds and convergence rates in terms of the discretization parameters. The bounds will suggest that, even though the algorithm converges for general space quantization, the performance of the learned value functions and policies will depend on the quantization scheme. Furthermore, provided error bounds will emphasize the effects of different levels of the approximation such as estimating the diffusion process with a controlled Markov chain, use of piece-wise constant policies, and state and action space quantization. Lastly, we will also discuss the effect of discretization parameters on the learning speed of the algorithm.


\subsection{Outline of the Paper}
In Section \ref{MC_section}, we construct a finite space Markov decision process (MDP) that will serve to approximate the diffusion process (\ref{diff}). In particular, in Section \ref{contMDP}, we present a controlled discrete time control process, which has the same distributions as the diffusion process at the sampling instances, when the diffusion process is controlled using piece-wise constant control functions. In Section \ref{finiteMDP}, we present a finite space controlled Markov chain, that is constructed based on the MDP from Section \ref{contMDP}, using state aggregation methods. 

In Section \ref{apprx}, we analyze the differences between the value function of the diffusion process (\ref{diff}), and the value function of the finite space MDP constructed in Section \ref{finiteMDP}. Furthermore, we provide upper bounds for the error (or regret) of the control policy designed for the finite space MDP, when it is used to control the diffusion process, in terms of the approximation parameters; where the comparison is with respect to the performance of the optimal admissible control process.

 In Section \ref{q_conv}, we show that the algorithm (\ref{QPOMDP}) converges, under certain assumptions, to the Q values of the finite space MDP constructed in Section \ref{finiteMDP}, and thus, we provide error bounds for the learned policy using the results from Section \ref{apprx}.

Finally, in Section \ref{sample}, towards a practical purpose, we analyze the convergence rate of the algorithm (\ref{QPOMDP}), with respect to the sampling interval lengths, $h$, and the number of iterations where we provide time complexity bounds.

\section{Approximate Markov Chains}\label{MC_section}
In this section, we provide two controlled Markov chains which will help us to analyze the error term (\ref{error}).

\subsection{A Markov Chain Construction with Exact Approximation of the Diffusion Process}\label{contMDP}
The first Markov chain we will present will have the same finite dimensional distributions with the sampled diffusion process under piece-wise constant control processes.

Let $\mathds{X}$ be the state space, and $\mathds{U}_h$ be the control action space of the Markov chain. 

We define the transition probabilities as follows: For any $k\in\Zplus$,  distribution of the state $X_k$ conditioned on the past state and action variables, is determined by the diffusion process (\ref{diff}), such that, conditioned on $(x_{k-1},\dots,x_0,u_{k-1},\dots,u_0)$, $X_k$ has the same distribution as 
\begin{align}\label{sampled_MDP}
X(h)=x_{k-1}+\int_0^{h}b(X(s),u_{k-1})ds+\int_0^h\sigma(X(s),u_{k-1})dB(s).
\end{align} 
Hence, for any $A\in\B(\mathds{X})$
\begin{align*}
Pr(X_k\in A|x_{[0,k-1]},u_{[0,k-1]})=\mathcal{T}_h(A|x_{k-1},u_{k-1})
\end{align*}
where ${(x,u)}_{[0,k-1]}:=x_0,\dots,x_{k-1},u_0,\dots,u_{k-1}$, such that 
\begin{align*}
\mathcal{T}_h(dx_k|x_{k-1},u_{k-1})\sim X(h)
\end{align*}
where $X(h)$ determined by (\ref{sampled_MDP}) and where  $\mathcal{T}_h$ is the transition kernel of the Markov chain which is a stochastic kernel from $\mathds{X}\times
\mathds{U}_h$ to $\mathds{X}$.

We also define a stage-wise cost function $c_h:\mathds{X}\times\mathds{U}_h\to \mathds{R}_+$ such that for any $(x,u)\in\mathds{X}\times\mathds{U}_h$
\begin{align*}
c_h(x,u):=c(x,u)\times h
\end{align*} where $c$ is the cost function of the diffusion process (see (\ref{cost})).

We now define the infinite horizon discounted cost function
\begin{align}\label{MDP1_cost}
  J_{\beta_h}(x_0,\gamma):= E_{x_0}^{\mathcal{T}_h,\gamma}\left[\sum_{k=0}^{\infty} \beta_h^k c_h(X_k,U_k)\right]
\end{align}
where $\beta_h:=e^{-\beta\times h}$, and $\gamma$ is an {\em admissible policy}. An {\em admissible policy} is a 
sequence of control functions $\{\gamma_k,\, k\in \Zplus\}$ such
that $\gamma_k$ is measurable with respect to the $\sigma$-algebra
generated by the information variables
$
I_k=\{X_{[0,k]},U_{[0,k-1]}\}, \quad k \in \mathds{N}, \quad
  \quad I_0=\{Y_0\},
$
where
\begin{equation}
\label{eq_control}
U_k=\gamma_k(I_k),\quad k\in \Zplus,\nonumber
\end{equation}
are the $\mathds{U}_h$-valued control actions.
\noindent We define $\Gamma$ to be the set of all such admissible policies. The optimal cost function is defined as
\begin{align}\label{MDP1_optcost}
J_{\beta_h}^*(x_0):= J_{\beta_h}(x_0,\gamma).
\end{align}

\begin{remark}
An important property of the MDP model we will make use of, is the following one: suppose that we are given an admissible policy $\gamma\in\Gamma$ defined for the MDP, we define the following control process $u(t)$ such that
\begin{align}\label{pccontrol}
u(t)=\gamma(X(k\times h)), \text{ for } t\in[k\times h, (k+1)\times h)
\end{align}
which is a piece-wise constant control process. Then, the controlled Markov chain state process $X_k$ under the policy $\gamma$ and the controlled diffusion process  $X(t)$ under the control process defined in (\ref{pccontrol}) have the same distributions at the sampling instances if they start from the same initial points, that is, for any $k\in \Zplus$, $$X_k\sim X(k\times h).$$
\end{remark}
\subsection{Finite State MDP Construction by Discretization of the State Space}\label{finiteMDP}

We now construct an MDP with a finite state space by dicretizing the state space $\mathds{X}$. 

We start by choosing a collection of disjoint sets $\{B_i\}_{i=1}^M$ such that $\cup_i B_i=\mathds{X}$, and $B_i\cap B_j =\emptyset$ for any $i\neq j$. Furthermore, we choose a representative state, $\hat{x}_i\in B_i$, for each disjoint set. For this setting, we denote the new finite state space by
$\mathds{X}_h:=\{\hat{x}_1,\dots,\hat{x}_M\}$. We put the dependence on the parameter $h$, since we will let the size of the finite set, $M$ go to $\infty$ as $h\to 0$. The mapping from the original state space $\mathds{X}$ to the finite set $\mathds{X}_h$ is done via
\begin{align}\label{quant_map}
\phi_{\mathds{X}}(x)=\hat{x}_i \quad \text{ if } x\in B_i.
\end{align}
Furthermore, we choose a weight measure $\pi^*(\cdot)\in\P(\mathds{X})$ on $\mathds{X}$ such that $\pi^*(B_i)>0$ for all $i\in\{1,\dots,M\}$. We now define normalized measures using the weight measure on each separate quantization bin $B_i$ such that  
\begin{align}\label{norm_inv}
\hat{\pi}_{\hat{x}_i}^*(A):=\frac{\pi^*(A)}{\pi^*(B_i)}, \quad \forall A\subset B_i, \quad \forall i\in \{1,\dots,M\}
\end{align}
 that is $\hat{\pi}_{\hat{x}_i}^*$ is the normalized weight measure on the set $B_i$, $\hat{x}_i$ belongs to. 

We now define the cost and transition kernels for this finite set using the normalized weight measures such that for any $\hat{x}_i,\hat{x}_j\in \mathds{X}_h$
\begin{align}\label{finite_cost}
&C_h^*(\hat{x}_i,u)=\int_{B_i}c_h(x,u)\hat{\pi}_{\hat{x}_i}^*(dx)\nonumber\\
&P_h^*(\hat{x}_j|\hat{x}_i,u)=\int_{B_i}\mathcal{T}_h(B_j|x,u)\hat{\pi}_{\hat{x}_i}^*(dx).
\end{align}
where $\mathcal{T}_h$ is the transition model for the MDP constructed in Section \ref{contMDP}.

Having defined the finite state space $\mathds{X}_h$, the cost function $C_h^*$ and the transition model $P_h^*$, we can now introduce the optimal value function for this finite model. We denote the optimal value function, which is defined on $\mathds{X}_h$, by $\hat{J}_{\beta_h}:\mathds{X}_h\to \mathds{R}$. Note that $\hat{J}_{\beta_h}$ satisfies the following Bellman equation for any $\hat{x}_i\in\mathds{X}_h$:
\begin{align}\label{finitepolicy}
\hat{J}_{\beta_h}(\hat{x}_i)=\inf_{u\in\mathds{U}_h}\left\{C_h^*(\hat{x}_i,u)+\beta_h\sum_{\hat{x}_1\in\mathds{X}_h}\hat{J}_{\beta_h}(\hat{x}_1)P_h^*(\hat{x}_1|\hat{x}_i,u)\right\}
\end{align}
We can easily extend this function over the state space $\mathds{X}$ by making it constant over the quantization bins. In other words, if $\hat{x}_i\in B_i$, where $\cup_i B_i=\mathds{X}$, for any $x\in B_i$, we write
\begin{align*}
\hat{J}_{\beta_h}(x):=\hat{J}_{\beta_h}(\hat{x}_i).
\end{align*}
Furthermore, the following equation follows directly from the dynamic programming principle for the Q values of the finite MDP:
\begin{align*}
Q_h^*(\hat{x}_i,u)=C^*_h(\hat{x}_i,u)+\beta_h\sum_{\hat{x}_1}\min_vQ_h^*(\hat{x}_1,v) P^*_h(\hat{x}_1|\hat{x}_i,u).
\end{align*}

\begin{remark}
We will prove that the iterations (\ref{QPOMDP}) converge to some $Q^*$ which satisfies the above equation, under suitable conditions on the diffusion process.
\end{remark}

We further define uniform error bounds resulting from the discretization of the state and actions spaces such that
\begin{align*}
L_{\mathds{X}}:=&\sup_{x\in\mathds{X}}\left|x-\phi_\mathds{X}(x)\right|,\\
L_{\mathds{U}}:=&\sup_{u\in\mathds{U}}\min_{u_h\in\mathds{U}_h}\left|u-u_h\right|.
\end{align*}
We note that later in the paper (see Corollary \ref{cor1}) we will see that, for the approximation error of the diffusion process to go to $0$ with increasing discretization rates, we will need $L_\mathds{X},L_{\mathds{U}}<h$.

\section{Approximation of the Diffusion Process by a Finite MDP}\label{apprx}

Recall that we are interested in the term
\begin{align*}
W_\beta(x_0,u_h)-W_\beta^*(x_0)
\end{align*}
where $u_h$ the control process obtained with $\gamma_h$ that is learned via $(\ref{QPOMDP})$. The first term represents the cost induced by the application of the control designed for the finite space MDP, when it is used for the continuous time model. The second term represents the optimal cost for the continuous time model. In Section \ref{q_conv}, we will prove that the learned policy $\gamma_h$ is optimal for the MDP defined in Section \ref{finiteMDP}. Hence, in this section, we will assume that $\gamma_h$ solves  (\ref{finitepolicy}).

We write the following:
\begin{align}
W_\beta(x_0,u_h)-W_\beta^*(x_0)&=W_\beta(x_0,u_h)-J_{\beta_h}(x_0,\gamma_h)\label{term1}\\
&\quad+J_{\beta_h}(x_0,\gamma_h)-J_{\beta_h}^*(x_0)\label{term2}\\
&\quad+J_{\beta_h}^*(x_0)-W_{\beta}^*(x_0).\label{term3}
\end{align}
In what follows, we will analyze each term separately.

\subsection{Analysis of term (\ref{term1})}
The following result provides a bound for the difference between the value function of a diffusion process controlled with piece-wise constant control processes and  the value function of an MDP controlled with an admissible policy.
\begin{assumption}\label{diff_assmp}
For the diffusion process given in (\ref{diff}), we assume that
\begin{itemize}
\item $\sup_{x,u}|b(x,u)|\leq B$, and $\sup_{x,u}|\sigma(x,u)|\leq B$, for some $B<\infty$.
\item $|b(x,u)-b(x',u')|+|\sigma(x,u)-\sigma(x',u')|\leq K\left(|x-x'|+|u-u'|\right)$ for some $K<\infty$ and for any $x,x'\in\mathds{X}$, and $u,u'\in\mathds{U}$.
\item $\sup_{x,u}|c(x,u)|\leq C$  for some $C<\infty$.
\item $|c(x,u)-c(x',u')|\leq K\left(|x-x'|+|u-u'|\right)$ for some $K<\infty$ and for any $x,x'\in\mathds{X}$, and $u,u'\in\mathds{U}$.
\item The process is nondegenerate such that $\sigma^2(x,u)>0$ for every $x,u\in\mathds{X}\times\mathds{U}$.

\end{itemize}
\end{assumption}
Note that these assumptions are sufficient for the existence of a unique strong solution to (\ref{diff}) under admissible control processes.

\begin{proposition}\label{time_app_thm}
Let $\gamma\in\Gamma$ be an admissible policy for the sampled controlled Markov chain, and $u_h(\cdot)$ be the corresponding piece-wise constant control process for the diffusion process such that
\begin{align*}
u_h(t)=\gamma(x(i\times h)), \text{ for } t\in[i\times h,(i+1)\times h].
\end{align*}
Under Assumption \ref{diff_assmp}, we have that for any $x_0\in\mathds{X}$
\begin{align*}
\left|W_\beta(x_0,u_h)-J_{\beta_h}(x_0,\gamma)\right|\leq hB+\frac{K B h}{1-e^{-\beta h}}\left(h+\sqrt{\frac{2h}{\pi}}\right).
\end{align*}
\end{proposition}
\begin{proof}
The proof can be found in Appendix \ref{time_app_thm_proof}.
\end{proof}

\subsection{Analysis of term (\ref{term2})}
For the analysis of (\ref{term2}),  we will need to calculate the Lipschitz constants of the sampled controlled Markov chain in terms of the Lipschitz constants of the diffusion process introduced in Assumption \ref{diff_assmp}.

In what follows, we will focus on the controlled Markov chain constructed in Section \ref{contMDP}. Recall that we have
\begin{align*}
\mathcal{T}_h(dx_k|x_{k-1},u_{k-1})\sim X(h)
\end{align*}
where
\begin{align*}
X(h)=x_{k-1}+\int_0^{h}b(X(s),u_{k-1})ds+\int_0^h\sigma(X(s),u_{k-1})dB(s).
\end{align*} 

Furthermore,
\begin{align*}
\beta_h&:=e^{-\beta h}\\
c_h(x,u)&:=c(x,u)\times h.
\end{align*}
\begin{lemma}\label{kernel_x}
Under Assumption \ref{diff_assmp},
\begin{align*}
W_1(\mathcal{T}_h(\cdot|y,u),\mathcal{T}_h(\cdot|x,u))\leq |x-y|e^{(K+\frac{K^2}{2})h}.
\end{align*}
where $W_1$ denotes the first order Wasserstein distance.
\end{lemma}
\begin{proof}
The proof can be found in Appendix \ref{kernel_x_proof}.
\end{proof}

\begin{lemma}\label{kernel_u}
Under Assumption \ref{diff_assmp}, if $h<1$
\begin{align*}
W_1(\mathcal{T}_h(\cdot|x,u),\mathcal{T}_h(\cdot|x,\hat{u}))\leq |u-\hat{u}|2K e^{2K^2h}.
\end{align*}
where $W_1$ denotes the first order Wasserstein distance.
\end{lemma}
\begin{proof}
The proof can be found in Appendix \ref{kernel_u_proof}.
\end{proof}

\begin{proposition}\label{quant_policy}
Under Assumption \ref{diff_assmp}, if $\beta>K+\frac{K^2}{2}$, and if $h<1$,
\begin{align*}
\sup_{x_0\in\mathds{X}}\left|J_{\beta_h}(x_0,\hat{\gamma}_h)-J^*_{\beta_h}(x_0)\right|\leq& \frac{Kh-Khe^{h(K+\frac{K^2}{2}-\beta)}+2K^2he^{h(2K^2-\beta)}}{(1-e^{-\beta h})(1-e^{h(K+\frac{K^2}{2}-\beta)})}L_\mathds{U}\\
&+ \frac{Kh}{(1-e^{-\beta h})(1-e^{h(K+\frac{K^2}{2}-\beta)})}L_\mathds{X}
\end{align*}  
\end{proposition}
\begin{proof}
The proof is a direct implication of \cite[Theorem 2.6]{kara2021qlearning}, which states that 

\begin{align*}
\sup_{x_0\in\mathds{X}}\left|J_{\beta_h}(x_0,\hat{\gamma}_h)-J^*_{\beta_h}(x_0)\right|\leq \frac{\alpha_c^{\mathds{X}}-\beta_h\alpha_\mathcal{T}^\mathds{X}\alpha_c^\mathds{U}+\beta_h\alpha_c^\mathds{X}\alpha_\mathcal{T}^\mathds{U}}{(1-\beta)(1-\beta_h\alpha_\mathcal{T}^\mathds{X})}L_\mathds{U}+\frac{\alpha_c^\mathds{X}}{(1-\beta_h)(1-\beta_h\alpha_\mathcal{T}^\mathds{X})}L_\mathds{X},
\end{align*}  
for constants  $\alpha_c^\mathds{X},  \alpha_c^\mathds{U},\alpha_\mathcal{T}^\mathds{X},\alpha_\mathcal{T}^\mathds{U}<\infty$ such that
\begin{itemize}
\item $|c(x,u)-c(x',u)|\leq \alpha_c^\mathds{X}|x-x'|$, 
\item $|c(x,u)-c(x,u')|\leq \alpha_c^\mathds{U}|u-u'|$,  
\item $W_1(\mathcal{T}(\cdot|x,u),\mathcal{T}(\cdot|x',u))\leq \alpha_\mathcal{T}^\mathds{X}|x-x'|$,
\item $W_1(\mathcal{T}(\cdot|x,u),\mathcal{T}(\cdot|x,u'))\leq \alpha_\mathcal{T}^\mathds{U}|u-u'|$,
\end{itemize}
Hence, the result follows from  Lemma \ref{kernel_x} and Lemma \ref{kernel_u}, by noting that $\alpha_c^\mathds{X},\alpha_c^\mathds{U}\leq h K$.
\end{proof}

\subsection{Analysis of term (\ref{term3})}
 Recall that (\ref{term3}) deals with the optimal value function of the sampled controlled Markov chain and the optimal value function of the controlled diffusion process. Hence, we make use of finite difference approximation methods for Bellman equations. The following result, taken from \cite{Krylov1999, jakobsen2019}, gives an upper bound on the performance loss of the piece-wise constant policies applied for the diffusion processes. 

\begin{lemma}[\cite{Krylov1999, jakobsen2019}]\label{constant_policy_bound}
Let $u_h^*$ denote the piece-wise constant control process, which is constant over the intervals $[kh,k(h+1))$ for $k\in\{0,1,\dots\}$, that achieves the minimum cost for (\ref{cost})  over such piece-wise constant policies. Under Assumption \ref{diff_assmp}, we have
\begin{align*}
W_\beta(x_0,u_h^*)-W_\beta^*(x_0)\leq N h^{\frac{1}{4}}
\end{align*}
for some constant $N<\infty$, which only depends on the discount factor $\beta$ and the Lipschitz coefficients of the model.
\end{lemma}

\begin{remark}\label{1ordererror}
One might expect that if the dynamics of the diffusion process are changing slowly, and the cost function does not have rapid changes with respect to the state process, then the time discretizetion leads to smaller performance losses. Indeed, as it is shown in \cite[Proposition 2.4]{jakobsen2019}, if $W_\beta$ and $c$ are regular enough, namely if 
\begin{align*}
\sup_u\bigg(\|L_u(L_uW_\beta)\|_\infty+\|L_uc\|_\infty\bigg)<\infty
\end{align*}
where $L_u$ is the generator function of the diffusion process for some control action $u$, then one might have
\begin{align*}
W_\beta(x_0,u_h^*)-W_\beta^*(x_0)\leq N h.
\end{align*}
\end{remark}

We are now ready to analyze (\ref{term3}).

\begin{proposition}\label{constant_policy}
Under Assumption \ref{diff_assmp}
\begin{align*}
J_{\beta_h}^*(x_0)-W_{\beta}^*(x_0) \leq hB+\frac{K B h}{1-e^{-\beta h}}\left(h+\sqrt{\frac{2h}{\pi}}\right)+N h^{\frac{1}{4}}.
\end{align*}
\end{proposition}

\begin{proof}
We start by the following bound
\begin{align*}
J_{\beta_h}^*(x_0)-W_{\beta}^*(x_0)&\leq J_{\beta_h}^*(x_0)-W_\beta(x_0,u_h^*)+W_\beta(x_0,u_h^*)-W_\beta^*(x_0)\\
&\leq J_{\beta_h}^*(x_0)-W_\beta(x_0,u_h^*)+ N h^{\frac{1}{4}}
\end{align*}
where the last bound follows directly from Theorem \ref{constant_policy_bound}. For the first term, let $\gamma_h^*$ denote the corresponding control policy for the controlled Markov chain ( that has the same law at the sampling instances as $u_h^*$). Note that this policy is not necessarily optimal for the sampled controlled Markov chain. Hence, we have the following bound:
\begin{align*}
 J_{\beta_h}^*(x_0)-W_\beta(x_0,u_h^*)&\leq J_{\beta_h}(x_0,\gamma_h^*)-W_\beta(x_0,u_h^*)\\
&\leq hB+\frac{K B h}{1-e^{-\beta h}}\left(h+\sqrt{\frac{2h}{\pi}}\right).
\end{align*}
The last bound follows from Theorem \ref{time_app_thm}. 

Combining what we have so far, we can conclude that
\begin{align*}
J_{\beta_h}^*(x_0)-W_{\beta}^*(x_0)\leq hB+\frac{K B h}{1-e^{-\beta h}}\left(h+\sqrt{\frac{2h}{\pi}}\right)+N h^{\frac{1}{4}}.
\end{align*}
\end{proof}

\subsection{Near Optimality of the Approximate Control}
Combining the results, we have presented so far, we can now state the main theorem of this section:
\begin{theorem}\label{main_thm1}
Under Assumption \ref{diff_assmp}, we have
\begin{align*}
W_\beta(x_0,u_h)-W_\beta^*(x_0)&\leq e(h,L_\mathds{X},L_\mathds{U})
\end{align*}
where 
\begin{align}\label{e_bound}
e(h,L_\mathds{X},L_\mathds{U})&:=2\left(hB+\frac{K B h}{1-e^{-\beta h}}\left(h+\sqrt{\frac{2h}{\pi}}\right)\right)\nonumber\\
&\quad+\frac{Kh-Khe^{h(K+\frac{K^2}{2}-\beta)}+2K^2he^{h(2K^2-\beta)}}{(1-e^{-\beta h})(1-e^{h(K+\frac{K^2}{2}-\beta)})}L_\mathds{U}\nonumber\\
&\quad+ \frac{Kh}{(1-e^{-\beta h})(1-e^{h(K+\frac{K^2}{2}-\beta)})}L_\mathds{X}\nonumber\\
&\quad+N h^{\frac{1}{4}}.
\end{align}
where $u_h$ is the control process obtained with $\gamma_h$, which is optimal for the finite MDP model constructed in Section \ref{finiteMDP}.
\end{theorem}

Note that, as stated in Remark \ref{1ordererror}, the last term in Theorem \ref{main_thm1} can be replaced by $Nh$, if the dynamics and the cost function are regular enough, e.g. if they do not change rapidly with time.

\begin{corollary}\label{cor1}
For small $h$, the upper bound derived in Theorem \ref{main_thm1}, behaves as
\begin{align*}
C\left(\sqrt{h}+\frac{L_\mathds{X}+L_\mathds{U}}{h}+h^{\frac{1}{4}}\right)
\end{align*}
for some $C<\infty$. This representation makes the distinction between the effects of  different steps of the approximation. The first term results from the Markov chain approximation of the diffusion process, the second term is due to the state and action space discretization, and finally the last term is due to the piece-wise constant control processes. We note again that the last term may be replaced with $h$ if the dynamics are regular enough.
\end{corollary}

\section{Convergence of the Learning Algorithm and Near Optimality of the Learned Policies}\label{q_conv}
In this section we present the main results of the paper. 

We first show that the iterations in (\ref{QPOMDP}) converge to the optimal Q values of the approximate controlled Markov chain constructed in Section \ref{finiteMDP}. Recall that the Q value iterations are in the following form:

\begin{align}\label{QPOMDP2}
Q_{k+1}(\hat{x},\hat{u})=(1-\alpha_k&(\hat{x},\hat{u}))Q_k(\hat{x},\hat{u})\nonumber\\
&+\alpha_k(\hat{x},\hat{u})\left(c(X(k\times h),\hat{u})\times h+\beta_h \min_{v\in\mathds{U}_h} Q_k\Big(\hat{X}_{k+1},v\Big)\right)
\end{align}
where $\beta_h=e^{-\beta\times h}$, and $\hat{X}_{k+1}$ is the discretized sampled state we observe following $\phi_\mathds{X}(x(i\times h))=\hat{x}$. Furthermore, $\phi_\mathds{X}$ maps the original state space $\mathds{X}$ to the finite subset $\mathds{X}_h$.

\begin{assumption}\label{learning_assmp}
\quad
\begin{itemize}
\item[]                i.  $\alpha_k(\hat{x},\hat{u})=0$ unless $(\hat{X}_k,u_k)=(\hat{x},\hat{u})$. Furthermore,
\[\alpha_k(\hat{x},\hat{u}) = {1 \over 1+ \sum_{t=0}^{k} 1_{\{\hat{X}_t=\hat{x}, u_t=\hat{u}\}} }.\]
This implies $\alpha_{k}(\hat{x},\hat{u})=\frac{1}{t+1}$ if we have visited $(\hat{x},\hat{u})$ pair $t$ many times until time $k$, i.e. the learning rates are linear.
\item[ ]             ii.  The controlled diffusion process converges to its unique invariant measure under the exploration policy.
\item[ ] iii. Every $\hat{x}\in\mathds{X}_h$ and $\hat{u}\in\mathds{U}_h$ is visited infinitely often during exploration. 
\end{itemize}
\end{assumption}

\begin{remark}
For the stability assumption (ii), we need the exploration policy to be a stabilizing policy such that it leads the process to its invariant measure. Since we use piece-wise constant policies for exploration, the stability can be tested using the Lyapunov type stability criteria for the resulting discrete time Markov decision process (see Section \ref{contMDP}). E.g. let $C\subset \mathds{X}$ be a compact set, $b\in \mathds{R}$, $\epsilon >0$, and $V:\mathds{X}\to \mathds{R}_+$ (e.g. $V(x)=|x|$), if the following is satisfied for all $x\in\mathds{X}$:
\begin{align*}
\int_{\mathds{X}}V(y)\mathcal{T}^\gamma_h(dy|x)=E[V(x_{t+1})|x_t=x]\leq V(x)-\epsilon+b\mathds{1}_{\{x\in C\}}
\end{align*}
then the process is positive Harris recurrent and thus admits a unique stationary measure (see \cite{MeynBook}), where $\mathcal{T}_h^\gamma$ is the transition kernel when we use $h$-rate time discretization under the piece-wise constant exploration policy $\gamma$.

The third assumption (iii) is a usual requirement for reinforcement learning algorithms.
\end{remark}

\begin{remark}
The stability assumption (ii), can be replaced with so called \it{replay buffer} (as in \cite{carvalho2020new}) such that if for every $k$, $X(k\times h)\sim\pi$, for some $\pi\in\P(\mathds{X})$, the iterations (\ref{QPOMDP2}) will converge. In particular, $\pi$ will take the role of the stationary distribution under the exploration policy.
\end{remark}

\begin{proposition}\label{Q_conv}
Under Assumption \ref{learning_assmp}, the iterations (\ref{QPOMDP2}) converge almost surely to some $Q_h^*:\mathds{X}_h\times\mathds{U}_h\to \mathds{R}$, which satisfies for every $\hat{x},\hat{u}\in \mathds{X}_h\times\mathds{U}_h$
\begin{align*}
Q_h^*(\hat{x},\hat{u})=C_h^*(\hat{x},\hat{u})+\beta_h\sum_{\hat{x}_1\in\mathds{X}_h}\min_{\hat{u}_1\in\mathds{U}_h}Q_h^*(\hat{x}_1,\hat{u}_1)P_h^*(\hat{x}_1|\hat{x},u)
\end{align*}
where $C_h^*$ and $P_h^*$ are defined in (\ref{finite_cost}).
\end{proposition}
\begin{proof}
The proof can be found in Appendix \ref{Q_conv_proof}.
\end{proof}

Once, $Q_h^*$ is obtained, one can construct the policies such that 
\[\gamma_h(\hat{x})=\argmin_{\hat{u}}Q_h^*(\hat{x},\hat{u})\]
Using these policies, following control processes are defined
\begin{align}\label{learned_control}
u_h(t)=\gamma_h\left(\phi_\mathds{X}(X(i\times h)\right), \text{ for } t\in [i\times h, (i+1)\times h).
\end{align}
Hence, $u_h$ is a piece-wise constant process, which changes value at the sampling instances according to the learned map $\gamma_h$.

The the following result is a direct implication of Theorem \ref{main_thm1}, and Theorem \ref{Q_conv}:
\begin{theorem}\label{main_thm}
Under Assumption \ref{diff_assmp}, and Assumption \ref{learning_assmp}, iterations in (\ref{QPOMDP2}) converge to some $Q_h^*$. For the learned control process $u_h$ (see (\ref{learned_control})), we have
\begin{align*}
W_\beta(x_0,u_h)-W_\beta^*(x_0)&\leq e(h,L_\mathds{X},L_\mathds{U})
\end{align*}
where $e(h,L_\mathds{X},L_\mathds{U})$ is defined in (\ref{e_bound}), and where $u_h$ is the control policy learned by using the approximate Q-learning algorithm (\ref{QPOMDP}).

Furthermore, for small $h$, we have that 
\begin{align*}
e(h,L_\mathds{X},L_\mathds{U})\leq  C\left(\sqrt{h}+\frac{L_\mathds{X}+L_\mathds{U}}{h}+h^{\frac{1}{4}}\right)
\end{align*}
for some $C<\infty$.
\end{theorem}

We now present results for the asymptotic case.

The first one states that if we first increase the quantization rate of the spaces and thus if the quantization error goes to 0, then if we take the sampling interval of the time to 0; the error bound goes to 0.
\begin{corollary}
Under Assumption \ref{diff_assmp},
\begin{align*}
\lim_{h\to 0}\lim_{L_\mathds{X},L_\mathds{U}\to 0}\left(W_\beta(x_0,u_h)-W_\beta^*(x_0)\right)&=0.
\end{align*}
\end{corollary}

We now define the quantization rate and the quantization error as a function of $h$ and denote them by $L_\mathds{X}(h),L_\mathds{U}(h)$:
\begin{corollary}
Under Assumption \ref{diff_assmp}, 
\begin{align*}
\lim_{h\to 0}\left(W_\beta(x_0,u_h)-W_\beta^*(x_0)\right)&=0
\end{align*}
if $L_\mathds{X}(h),L_\mathds{U}(h)$ go to 0 at a faster rate than $h$, i.e. if $\frac{L_\mathds{X}(h)}{h},\frac{L_\mathds{U}(h)}{h}\to 0$ as $h \to 0$.
\end{corollary}

\begin{remark}
The space discretization approach we have followed so far gives us precise error bounds and a convergence analysis under general conditions. However, we can adapt Q learning algorithms with function approximation to our setting as well.
In particular, we will focus on linear approximations and discuss their convergence properties. Consider the set of Q functions that can be parametrized over the parameter $\theta$ that can be expressed as the linear span of a fixed set of M linearly independent functions $\phi_i : \mathds{X} \times \mathds{U} \to \mathds{R}$, such that the Q values can be written as 
\begin{align*}
Q_{\theta}(x,u)=\sum_{i=1}^M\phi_i(x,u)\theta(i),
\end{align*} 
then we can construct the following iterations to learn the Q values over the parametrized family:
\begin{align}\label{lin_q_1}
\theta_{k+1}=\theta_k+\alpha_k \phi(\hat{x}_k,u_k)\Delta_k
\end{align}
where $\Delta_t:=c(\hat{x}_t,u_t)+\beta_h \max_{u\in\mathds{U}}Q_t(\hat{X}_{t+1},u)-Q_t(\hat{x}_t,u_t)$, such that $\hat{x}_k$ is the sampled diffusion process. 

Then the convergence of this iterations can be shown (see \cite{melo2008analysis}). In particular, one can show that the algorithm converges if 
\begin{itemize}
\item The sampled process converges to its stationary distribution during exploration under the exploration policy,
\item The exploration policy is already close to the optimal policy (precise condition can be found in \cite{melo2008analysis}).
\end{itemize}
Note that the second assumption is quite restrictive. Furthermore, \cite{melo2008analysis} does not provide an analysis for the error analysis of the limit Q function with respect to the optimal Q values, which in our setting would be the optimal Q values of the MDP model constructed in Section 2.1 using piece-wise constant controls for the diffusion process. 

In  a recent work (\cite{carvalho2020new}), the assumptions are relaxed using `coupled Q learning' with a so called `replay buffer' assumption. If one uses the following iterations:
\begin{align*}
&u_{t+1}=u_t+\alpha_t(\phi(x_t,u_t)Q_{v_t}(x_t,u_t)-u_t),\\
&v_{t+1}=v_t+\beta_t\phi(x_t,u_t)\Delta_t
\end{align*}
where $\Delta=c(x_t,u_t)+\beta\max_{u\in\mathds{U}}Q_{u_t}(X_{t+1},u)-Q_{v_t}(x_t,u_t)$. It is then shown that these iterations converge if 
\begin{itemize}
\item $\alpha_t$ and $\beta_t$ are square summable but not summable, and $\alpha_t=o(\beta_t)$ or $\alpha_t\ll \beta_t$,
\item For all $t$, $(x_t)$ can be sampled from a fixed distribution, say  $\pi$, or so called {\it replay buffer}.
\end{itemize}
Furthermore, for the limit Q values $Q_{v^*}$, we have that
\begin{align}\label{lin_q_err}
\|Q^*-Q_{v^*}\|_\infty \leq \frac{1}{1-\beta}\|Q^*-\text{Proj}_{\phi}Q^*\|_\infty +\frac{1-\sigma}{\sigma}\frac{\beta\sigma}{(1-\beta)^2}
\end{align}
where $\sigma$ is constant that depends on the set of linear basis functions. Hence, if the the optimal Q values are in the linear span of basis functions, we might get an approximation error depending on the basis functions. Note further that the replay buffer assumption replaces the stationarity assumption. However, existence of such a setup might be hard to find, i.e. one may not be able to start the process from a desired distribution.

In summary, both of these Q learning with linear approximations algorithms can be used by discretizing the time for diffusion processes. Under somehow restrictive assumptions, convergence can also be shown. Furthermore, the approximation error will be in the order of $h^{1/4}$ (piece-wise constant policy approximation error) plus the error presented in (\ref{lin_q_err}). However, both of these results are still not fully conclusive, as they do not analyze the performance of the learned policies but only focus on the difference between the limit Q values and the optimal Q values.
\end{remark}

\section{A Discussion on the Convergence Rate of the Learning and the Effect of Sampling and Quantization Rates on the Learning Speed}\label{sample}
For approximation accuracy, finer sampling intervals and higher quantization rates lead to smaller error bounds. However, it is clear that finer sampling intervals and higher quantization rates also result in slower learning. In particular, higher quantization rate of the state and action spaces results in larger aggregate state and action spaces which in turn leads to dimension issues for the learning. Furthermore, finer sampling intervals increases the discount factor $\beta_h$ of the approximate Q learning algorithm in (\ref{QPOMDP2}), and higher discount rates make the iterations in (\ref{QPOMDP2}) to converge at a slower rate by increasing the effective horizon of the problem. 

First, we note that the algorithm presented in this paper relies on the convergence of the state process to its stationary distribution, since the state aggregation results in non-Markovian dynamics. Hence, the convergence speed of the algorithm depends on the convergence to the invariant measure of the process. Our motivation in this section is to study the effect of state and time discretization parameters on the speed of the Q learning algorithm. However, these discretization parameters do not affect the speed of the convergence to the stationary distribution, as this is related to the dynamics of the underlying diffusion process. Thus, for a simpler presentation, we will assume the state process starts from its stationary distribution under the exploration policy and always stays there during the exploration. 


The following well known result (\cite[Theorem 5]{even2003learning}) is stated using the notation of this paper. The result provides a sample complexity bound for the near optimal Q estimates when the learning rate is linear, i.e. $\alpha_k=\frac{1}{k}$, as in this paper.

\begin{proposition}[\cite{even2003learning}]\label{even_thm}
Let $Q_T$ be the value of the Q-learning algorithm using linear learning rate at time $T$. Then with probability at least $1-\delta$, for any positive constant $\psi$ we have $\|Q_T-Q_h^*\|\leq \epsilon$, given that 
\begin{align}\label{bound1}
T=\Omega\left((L+\psi L+1)^{\frac{2}{1-\beta_h}\text{ln}\frac{V_{\text{max}}}{\epsilon}}\frac{V_{\text{max}}^2\text{ln}(\frac{|\mathds{X}_h||\mathds{U}_h|V_{\text{max}}}{\delta\epsilon \psi (1-\beta_h)})}{(\psi \epsilon (1-\beta_h))^2}\right),
\end{align}
where $V_{\text{max}}:=\|\hat{J}_{\beta_h}\|_\infty\leq \frac{h\|c\|_\infty }{1-\beta_h}$, and $L$ is the covering time for the algorithm, that is the smallest time for every state and action pair to be visited at least once. 
\end{proposition}

The above result can be simplified for small enough $h$. We can write that 
\begin{align}\label{samp}
T \lesssim  \left(\frac{1}{\epsilon}\right)^{\frac{\text{ln}(L+\psi L+1)}{h}}\frac{\text{ln}(\frac{|\mathds{X}_h||\mathds{U}_h|}{\delta \epsilon \psi h})}{\epsilon^2h^2},
\end{align}
where we use $\lesssim$ as we drop some constant and logarithmic dependence. 

From (\ref{samp}), we observe that decreasing the time discretization parameter $h$, increases the sample complexity exponentially. Note that, as the required sample size increases exponentially, the duration we need to observe the diffusion process in real time to get $\epsilon$-near estimates also increases, since the exponential increase in the sample complexity dominates the decrease rate on $h$.

For the effect of the space discretization, we can see that the sample complexity increases in a logarithmic way with the second term, but the dominant effect is caused by the increase on the cover time $L$, as we have that $L\geq |\mathds{X}_h|\times |\mathds{U}_h|$, and depending on the sampling frequency of the samples, the cover time can be even greater. Nonetheless, the cover time increases at least linearly with the increase on the size of the aggregate state and action spaces, which in turn increases the sample complexity at a polynomial rate depending on the parameter $h$.

Figure \ref{t_eps_h} shows the change on the required time when for $h\in[0.7,0.9]$, $\epsilon\in[0.1,0.5]$ and when $|\mathds{X}_h|$ and $|\mathds{U}_h|$ are assumed to be order of $\frac{1}{h}$.

\begin{figure}[h!]
\begin{center}
\includegraphics[scale=0.6]{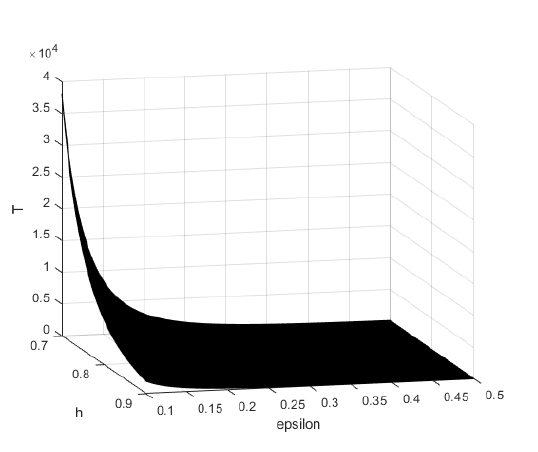}
\caption{Sample complexity for different values of error and sampling parameters }
 \label{t_eps_h}
\end{center}
\end{figure}
However, this is the sample complexity to achieve $\epsilon$-near estimates of $Q_h^*$ which is also an estimate of the true value function (see Theorem \ref{main_thm} and Corollary \ref{cor1}).  Hence, one needs to pick the sampling parameter $h$ in a careful way considering the trade-off between the convergence speed and the approximation accuracy. Following Figure \ref{error_bounds_T} provides an example for the error bounds given a given level sample points $T\sim 9500$. The graph on the left shows the difference between $Q_T$ and $Q_h^*$ (see (\ref{QPOMDP2})), clearly the error for the learning of the approximate model decays as $h$ increases for a fixed level of sample points since the learned model becomes simpler as $h$ increases.  The graph on the right shows represents the upper bound on the difference
\begin{align*}
|Q_T-Q^*|\leq |Q_T-Q_h^*|+|Q_h^*-Q^*|,
\end{align*}
with proper scaling. Note that the first term is the distance from the approximate Q value, whereas the second term is the approximation error. We can see that after a certain value, increasing $h$ results on greater total error for these specific parameter intervals.
\begin{figure}[h!]
\begin{center}
\includegraphics[scale=0.35]{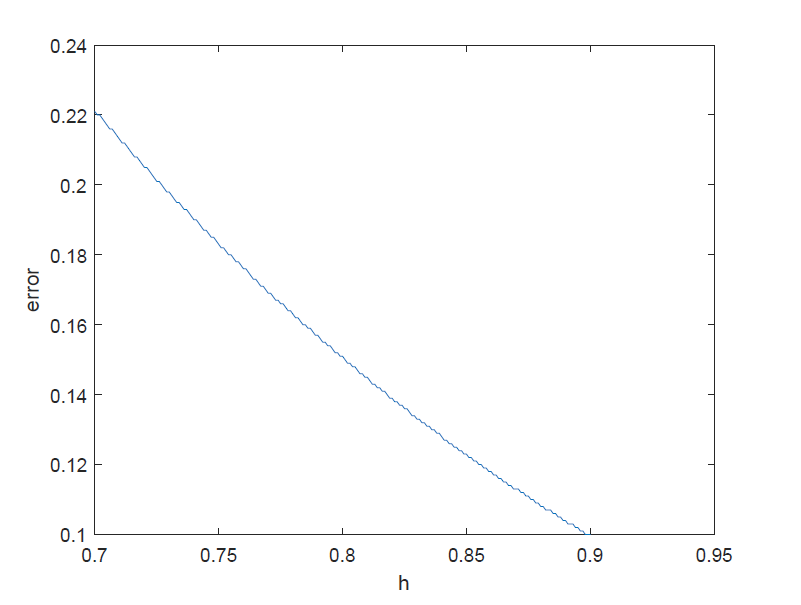}\includegraphics[scale=0.35]{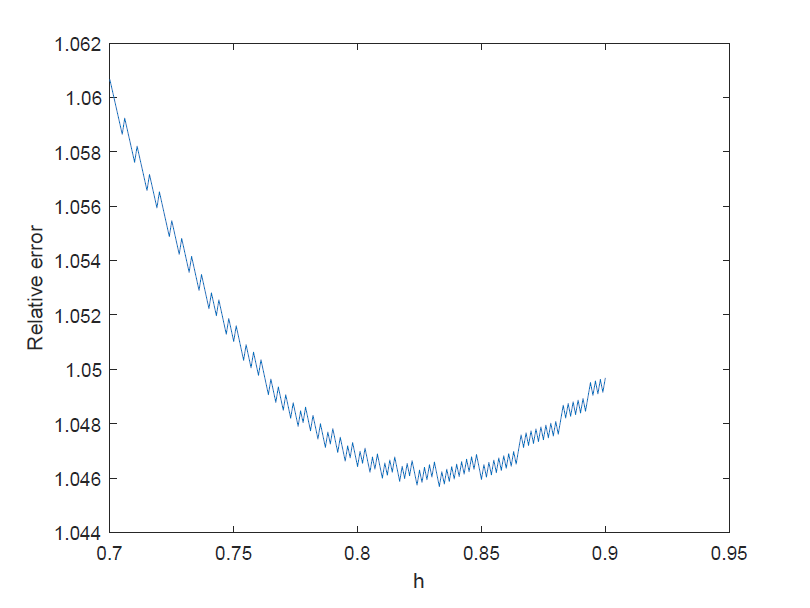}
\caption{$|Q_T^h-Q_h^*|$ and $|Q_T^h-Q^*|$ for different values of $h$ when $T\sim 9500$.}
  \label{error_bounds_T}
\end{center}
\end{figure}

The sample complexity bound presented in (\ref{bound1}) reveals that the increase is exponential in $\frac{1}{1-\beta_h}$, this rate is  clearly not desired and it turns out that it can be avoided using different learning rates, $\alpha_k$, rather than using linear learning rates. For example, using polynomial learning rates, $\alpha_k=\frac{1}{k^\omega}$ for some $\omega\in(1/2,1)$, one can achieve following sample complexity for $\epsilon$-near estimates for small enough $h$ (see \cite[Theorem 4]{even2003learning}):
\begin{align*}
T\lesssim \left(\frac{L^{1+3\omega}\text{ln}(|\mathds{X}_h||\mathds{U}_h|)}{h^2\epsilon^2}\right)^{\frac{1}{\omega}}+\left(\frac{L}{h}\text{ln}(\frac{1}{\epsilon})\right)^{\frac{1}{1-\omega}},
\end{align*}
where the exponential increase with respect to $h$ is eliminated.

Further improvements and variations can be achieved for the sample complexity and the convergence rate, using different learning rates, e.g. rescaled learning rates, carefully chosen constant learning rates, or with different variations of the Q learning algorithm such as speedy Q learning, or variance reduced Q learning (see e.g. \cite{Azar2011SpeedyQ,wainwright2019stochastic,li2020sample}).

We note that, even though different convergence rates can be derived using different reinforcement learning algorithms, if the algorithms are constructed using the time and state discretization procedure as in this paper, the learned value function will be the value function of the approximate MDP model constructed in Section \ref{finiteMDP}. Hence, the convergence speed can be improved with different learning rates or different Q learning variants, however,
\begin{align*}
\sqrt{h}+\frac{L_{\mathds{X}}+L_{\mathds{U}}}{h}+h^{\frac{1}{4}}
\end{align*}
which is the error upper-bound for the approximation via state and time discretization and piece-wise constant polices, will not change.

\section{Conclusion}
We have constructed an approximate Q learning algorithm for a controlled diffusion process through discretization in time and space. We have showed that this algorithm converges under an ergodicity assumption on the state process. Furthermore, we have showed that the limit Q values satisfy the optimality equation of a finite Markov decision process, which has the same distribution as the diffusion process at the sampling points when the diffusion process is controlled with a piece-wise constant control process. Using these observations, we have derived upper bounds, as a function of the discretization parameters, for the approximation error of the learned policies compared to the performance of the optimal admissible control process. 

Possible future directions, building on the analysis in this paper are as follows:
\begin{itemize}
\item When we discretize the state space, the aggregated state process is no longer a Markov process, hence, we use an ergodicity assumption, to guarantee the convergence of the Q learning algorithm. However, if the quantization is fine enough, one might expect the learning algorithm to stay in a set of values with sufficiently small variations even without the ergodicity assumption. Hence, a possible future problem is to relax the ergodicity assumption we consider here.

\item In this paper, we have used the traditional Q learning algorithm, for the simplicity of the presentation. For faster learning rates, different variations such as  the variance reduction techniques, can be considered. We note that, if one uses the same time and space quantization scheme, the learned value functions will be the same as in this paper, however, using different variations of the Q learning algorithm will change the learning speed.
\item The analysis used in this paper, can be extended to partially observed systems using the results from \cite{kara2021convergence,kara2020near} under proper filter stability conditions.
\item We have not considered the exploration and exploitation trade-off; one might study this relation considering the provided convergence rate in this paper.
\end{itemize}

\newpage

\appendix

\section{Proof of Proposition \ref{time_app_thm}}\label{time_app_thm_proof}
Note that the way we constructed the controlled Markov chain in Section \ref{contMDP}, implies that the state process for the controlled Markov chain, and the state process for the diffusion process have the same distribution at the sampling instances. That is 
\begin{align*}
X_k\sim X(k\times h), \text{ for all } k\in\mathds{Z}_+.
\end{align*}
We can then write that
\begin{align*}
&\left|W_\beta(x_0,u_h)-J_{\beta_h}(x_0,\gamma)\right|=\left|E\left[\int_0^\infty e^{-\beta s}c\left(X(s),u_h(s)\right)ds\right]-E\left[\sum_{k=0}^\infty \beta_h^kc_h(X_k,U_k)\right]\right|\\
&=\left|E\left[\sum_{k=0}^\infty\int_{kh}^{(k+1)h}e^{-\beta s}c(X(s),u_h(s))ds\right]-E\left[\sum_{k=0}^\infty \beta_h^kc_h(X_k,U_k)\right]\right|\\
&\leq \sum_{k=0}^\infty\int_{kh}^{(k+1)h}E\left[\left|e^{-\beta s}c(X(s),u_h(s))-e^{-\beta h k}\frac{c_h(X_k,U_k)}{h}\right|\right]ds
\end{align*}
We know focus on the term inside:
\begin{align*}
&E\left[\left|e^{-\beta s}c(X(s),u_h(s))-e^{-\beta h k}\frac{c_h(X_k,U_k)}{h}\right|\right]\\
&\leq E\left[\left|e^{-\beta s}c(X(s),u_h(s))-e^{-\beta h k}c(X(s),u_h(s))\right|\right]+E\left[\left|e^{-\beta h k}c(X(s),u_h(s))-e^{-\beta h k}c(X_k,U_k)\right|\right]\\
&\leq e^{-\beta h k}(1- e^{-\beta h})\|c\|_\infty+ e^{-\beta h k} K E\left[ |X(s)-X_k|\right]
\end{align*}
for the last step, we used the fact that $u_h(s)=U_k$ as it is a piece-wise constant control process. For the second term, we have that for $s\in[kh,(k+1)h)$
\begin{align*}
X(s)=X_k+\int_{kh}^{s}b(x(r),u_h(r))dr+\int_{kh}^s\sigma(x(r),u_h(r))dB(r).
\end{align*}
Thus, using Assumption \ref{diff_assmp} we can write
\begin{align*}
 E\left[ |X(s)-X_k|\right]&\leq \|b\|_\infty \int_{kh}^{s}dr+\|\sigma\|_\infty E\left[\left|\int_{kh}^s dB(r)\right|\right]\\
&\leq B h + B E[|Z_h|]=B \left(h+ \sqrt{\frac{2h}{\pi}}\right)
\end{align*}
where $Z_h$ is normally distributed with mean 0 and  variance $h$. 

By combining everything we have so far, we write
\begin{align*}
&\left|W_\beta(x_0,u_h)-J_{\beta_h}(x_0,\gamma)\right|\\
&\leq \sum_{k=0}^\infty\int_{kh}^{(k+1)h} e^{-\beta h k}(1- e^{-\beta h})\|c\|_\infty + e^{-\beta h k} K B \left(h+ \sqrt{\frac{2h}{\pi}}\right) ds\\
&=  \sum_{k=0}^\infty\left( e^{-\beta h k}(1- e^{-\beta h})\|c\|_\infty + e^{-\beta h k} K B \left(h+ \sqrt{\frac{2h}{\pi}}\right) \right) h\\
&=hB+\frac{K B h}{1-e^{-\beta h}}\left(h+\sqrt{\frac{2h}{\pi}}\right)
\end{align*}

\section{Proof of Proposition \ref{Q_conv}}\label{Q_conv_proof}
\begin{proof}

We start by writing the iterations in a more compact from by defining $C_h:=c(x(i\times h),\hat{u})\times h$, and $\hat{X}_1:=\phi_\mathds{X}\Big(X\big((i+1)\times h\big)\Big)$:
\begin{align*}
Q_{k+1}(\hat{x},\hat{u})=(1-\alpha_k(\hat{x},\hat{u}))Q_k(\hat{x},\hat{u})+\alpha_k(\hat{x},\hat{u})\left(C_h+\beta_h \min_{v\in\mathds{U}_h} Q_k\Big(\hat{X}_1,v\Big)\right)
\end{align*}

We  define 
\begin{align*}
\Delta_k(\hat{x},\hat{u})&:=Q_k(\hat{x},\hat{u})-Q_h^*(\hat{x},\hat{u})\\
F_k(\hat{x},\hat{u})&:=C_h+\beta_h V_k(\hat{X}_1)-Q_h^*(\hat{x},\hat{u})\\
\hat{F}_k(\hat{x},\hat{u})&:=C^*_h(\hat{x},\hat{u})+\beta_h\sum_{\hat{x}_1} V_k(\hat{x}_1)P_h^*(\hat{x}_1|\hat{x},\hat{u}) -Q_h^*(\hat{x},\hat{u}),
\end{align*}
where $V_k(\hat{x}):=\min_{v\in\mathds{U}_h} Q_k\Big(\hat{x},v\Big)$.

Then, we can write the following iteration
\begin{align*}
\Delta_{k+1}(\hat{x},\hat{u})=(1-\alpha_k(\hat{x},\hat{u}))\Delta_k(\hat{x},\hat{u})+\alpha_k(\hat{x},\hat{u}) F_k(\hat{x},\hat{u}).
\end{align*}
Now, we write $\Delta_k=\delta_k+w_k$ such that 
\begin{align*}
\delta_{k+1}(\hat{x},\hat{u})&=(1-\alpha_k(\hat{x},\hat{u}))\delta_k(\hat{x},\hat{u})+\alpha_k(\hat{x},\hat{u}) \hat{F}_k(\hat{x},\hat{u})\\
w_{k+1}(\hat{x},\hat{u})&=(1-\alpha_k(\hat{x},\hat{u}))w_k(\hat{x},\hat{u})+\alpha_k(\hat{x},\hat{u}) r_k(\hat{x},\hat{u})
\end{align*}
where $r_k:=F_k-\hat{F}_k=\beta V_k(\hat{X}_1)-\beta_h \sum_{\hat{x}_1}V_k(\hat{x}_1)P_h^*(\hat{x}_1|\hat{x},\hat{u}) + C_h- C_h^*(\hat{x},\hat{u})$. Next, we define
\begin{align*}
r_k^*(\hat{x},\hat{u})=\beta_h V^*(\hat{X}_1)-\beta_h \sum_{\hat{x}_1}V^*(\hat{x}_1)P_h^*(\hat{x}_1|\hat{x},\hat{u}) + C_h-C_h^*(\hat{x},\hat{u})
\end{align*}
We further separate $w_k=u_k+v_k$ such that
\begin{align*}
u_{k+1}(\hat{x},\hat{u})&=(1-\alpha_k(\hat{x},\hat{u}))u_k(\hat{x},\hat{u})+\alpha_k(\hat{x},\hat{u}) e_k(\hat{x},\hat{u})\\
v_{k+1}(\hat{x},\hat{u})&=(1-\alpha_k(\hat{x},\hat{u}))v_k(\hat{x},\hat{u})+\alpha_k(\hat{x},\hat{u}) r^*_k(\hat{x},\hat{u})
\end{align*}
where $e_k=r_k-r^*_k$. 



We now show that $v_k(\hat{x},\hat{u})\to 0$ almost surely for all $(\hat{x},\hat{u})$. Note that, because of the way we chose the learning rates $\alpha_k$, $v_k(\hat{x},\hat{u})$ is only updated when the process hits $\hat{x},\hat{u}$ during the exploration. Thus, we define the following stopping times
\begin{align*}
\tau(n+1)=\{\min k>\tau(n): \phi_\mathds{X}(x_k)=\hat{x},u_k=\hat{u}\}
\end{align*}
where $\tau(0)=0$. In what follows, we will focus on the $v_n(\hat{x},\hat{u})$ process, since we assume that every $(\hat{x},\hat{u})$ pair is visited infinitely often, these stopping times are bounded almost surely and hence, we can make sure that $n\to \infty$ as $k\to \infty$. Furthermore, we have that $\alpha_n(\hat{x},\hat{u})=\frac{1}{n}$, as $n$ is the number of times the we have hit $(\hat{x},\hat{u})$ pair.

When $\alpha_n(\hat{x},\hat{u})=\frac{1}{n}$ for every $(\hat{x},\hat{u})$ pair, the problem reduces to
\begin{align*}
v_{n+1}(\hat{x},\hat{u})&=\frac{1}{n}\sum_{n'=0}^{n-1} r^*_{n'}(\hat{x},\hat{u}).
\end{align*}

Above, even though, we do not write the time dependence on $\hat{X}_1$, distribution of the $\hat{X}_1$ is different every time we make the update, since the marginal distribution of $x_t$ is different. However, using the fact that the random variables $(\hat{X}_1,X_1)$ form a controlled Markov chain since they are sampled at the stopping times $\tau(k)$, and using the invariant measure of the original state process,  we can write 
\begin{align*}
&\lim_{n\to\infty}\frac{1}{n}\sum_{n'=0}^{n-1}V^*(\hat{X}_1)=\int_{B}\sum_{j}V^*(\hat{x}_j)\mathcal{T}_h(B_j|x,\hat{u})\hat{\pi}_{\hat{x}}^*(dx)\\
&\lim_{n\to\infty}\frac{1}{n}\sum_{n'=0}^{n-1}C_h=\lim_{n\to\infty}\frac{1}{n}\sum_{n'=0}^{n-1}c(x(\tau(n')\times h),\hat{u})\times h=\int_{B}c_h(x,u)\hat{\pi}_{\hat{x}}^*(dx).
\end{align*} 
where $B$ is the quantization bin $\hat{x}$ belongs to. Hence, we have proved that $v_k(\hat{x},\hat{u})\to 0$, noting
\begin{align*}
&C_h^*(\hat{x}_i,u)=\int_{B_i}c_h(x,u)\hat{\pi}_{\hat{x}_i}^*(dx)\nonumber\\
&P_h^*(\hat{x}_j|\hat{x}_i,u)=\int_{B_i}\mathcal{T}_h(B_j|x,u)\hat{\pi}_{\hat{x}_i}^*(dx).
\end{align*}

Now, we go back to the iterations:
\begin{align*}
\delta_{k+1}(\hat{x},\hat{u})&=(1-\alpha_k(\hat{x},\hat{u}))\delta_k(\hat{x},\hat{u})+\alpha_k(\hat{x},\hat{u}) \hat{F}_k(\hat{x},\hat{u})\\
u_{k+1}(\hat{x},\hat{u})&=(1-\alpha_k(\hat{x},\hat{u}))u_k(\hat{x},\hat{u})+\alpha_k(\hat{x},\hat{u}) e_k(\hat{x},\hat{u})\\
v_{k+1}(\hat{x},\hat{u})&=(1-\alpha_k(\hat{x},\hat{u}))v_k(\hat{x},\hat{u})+\alpha_k(\hat{x},\hat{u}) r^*_k(\hat{x},\hat{u}).
\end{align*}
Note that, we want to show $\Delta_k=\delta_k+u_k+v_k \to 0$ almost surely and we have that $v_k(\hat{x},\hat{u})\to 0$ almost surely for all $(\hat{x},\hat{u})$. The following analysis holds for any path that belongs to the probability one event in which $v_k(I,u)\to 0$. For any such path and for any given $\epsilon>0$, we can find an $N<\infty$ such that $\|v_k\|_\infty<\epsilon$ for all $k>N$ as $(I,u)$ takes values from a finite set.

We now focus on the term $\delta_k + u_k$ for $k>N$:
\begin{align}\label{sum_proc}
(\delta_{k+1}+u_{k+1})(\hat{x},\hat{u})&=(1-\alpha_k(\hat{x},\hat{u}))(\delta_k+u_k)(\hat{x},\hat{u})+\alpha_k(\hat{x},\hat{u}) (\hat{F}_k+e_k)(\hat{x},\hat{u}).
\end{align}
Observe that  for $k>N$,
\begin{align*}
(\hat{F}_k+e_k)(\hat{x},\hat{u})=&(F_k-r_k^*)(\hat{x},\hat{u})=\beta_h V_k(\hat{X}_1)-\beta_h V^*(\hat{X}_1)\\
&\leq \beta_h\max_{\hat{x},\hat{u}}|Q_k(\hat{x},\hat{u})-Q^*(\hat{x},\hat{u})|=\beta_h\|\Delta_k\|_\infty\\
&\leq \beta_h \|\delta_k+u_k\|_\infty+\beta_h \epsilon
\end{align*}
where the last step follows from the fact that $v_k\to 0$ almost surely. By choosing $C<\infty$ such that $\hat{\beta}:=\beta_h(C+1)/C<1$, for $\|\delta_k+u_k\|_\infty>C\epsilon$, we can write that
\begin{align*}
 \beta_h \|\delta_k+u_k+\epsilon\|_\infty\leq \hat{\beta}\|\delta_k+u_k\|_\infty.
\end{align*}
Now we rewrite (\ref{sum_proc})
\begin{align*}
(\delta_{k+1}+u_{k+1})(\hat{x},\hat{u})&=(1-\alpha_k(\hat{x},\hat{u}))(\delta_k+u_k)(\hat{x},\hat{u})+\alpha_k(\hat{x},\hat{u}) (\hat{F}_k+e_k)(\hat{x},\hat{u})\\
&\leq (1-\alpha_k(\hat{x},\hat{u}))(\delta_k+u_k)(\hat{x},\hat{u})+\alpha_k(\hat{x},\hat{u})  \hat{\beta}\|\delta_k+u_k\|_\infty.
\end{align*}
By \cite[Lemma 3]{jaakkola1994convergence}, $(\delta_{k+1}+u_{k+1})(\hat{x},\hat{u})$ tends to $0$ for $\|\delta_k+u_k\|_\infty>C\epsilon$. This shows that the condition $\|\delta_k+u_k\|_\infty>C\epsilon$ cannot be sustained indefinitely. Next, we show that once the process hits below $C\epsilon$ it always stays there. Suppose $\|\delta_k+u_k\|_\infty<C\epsilon$,
\begin{align*}
(\delta_{k+1}+u_{k+1})(\hat{x},\hat{u})&\leq(1-\alpha_k(\hat{x},\hat{u}))(\delta_k+u_k)(\hat{x},\hat{u})+\alpha_k(\hat{x},\hat{u}) \beta_h \left(\|\delta_k+u_k\|_\infty+\epsilon\right)\\
&\leq (1-\alpha_k(\hat{x},\hat{u})) C\epsilon + \alpha_k(\hat{x},\hat{u}) \beta_h (C\epsilon + \epsilon)\\
&= (1-\alpha_k(\hat{x},\hat{u})) C\epsilon  + \alpha_k(\hat{x},\hat{u}) \beta_h (C+1)\epsilon\\
&\leq  (1-\alpha_k(\hat{x},\hat{u})) C\epsilon  + \alpha_k(\hat{x},\hat{u}) C\epsilon, \quad (\beta_h(C+1)\leq C)\\
&=C\epsilon.
\end{align*}
Then, we can write $\|\delta_{k+1}+u_{k+1}\|_\infty<C\epsilon$. 

Thus, taking $\epsilon \to 0$, we can conclude that $\Delta_k=\delta_k+u_k+v_k \to 0$ almost surely.

Therefore, the process $Q_k$, determined by the algorithm converges almost surely to $Q_h^*$. 
\end{proof}

\section{Proof of Lemma \ref{kernel_x}}\label{kernel_x_proof}
We are interested in the distance between the distributions of the following random variables
\begin{align*}
&x(h)=x(0)+\int_0^{h}b(x(s),u)ds+\int_0^h\sigma(x(s),u)dB(s)\\
&y(h)=y(0)+\int_0^{h}b(y(s),u)ds+\int_0^h\sigma(y(s),u)dB(s).
\end{align*}
We define $z(h):=x(h)-y(s)$, whose dynamics are given by
\begin{align*}
z(h)=z(0)+\int_0^h \bar{b}(x(s),y(s),u)ds+\int_0^h \bar{\sigma}(x(s),y(s),u)dB(s)
\end{align*}
where 
\begin{align*}
\bar{b}(x(s),y(s),u)&:=b(x(s),u)-b(y(s),u)\\
\bar{\sigma}(x(s),y(s),u)&:=\sigma(x(s),u)-\sigma(y(s),u).
\end{align*}
Using the Ito formula 
\begin{align*}
E[Z(h)^2]&\leq Z_0^2+E\left[\int_0^h 2Z_t\bar{b}(x(s),y(s),u)ds \right]+ E\left[\int_0^h \bar{\sigma}^2(x(s),y(s),u)ds  \right]\\
&\leq  Z_0^2 +2K \int_0^hE\left[Z_s^2\right]ds+ K^2 \int_0^hE\left[Z_s^2\right]ds.
\end{align*}
We can then use Gronwall's inequality to write
\begin{align*}
E[Z(h)^2]\leq Z_0^2 e^{(2K+K^2)h}.
\end{align*}
Using the Holder's inequality, we can further write
\begin{align*}
E[|X(h)-Y(h)|]\leq |x(0)-y(0)|e^{(K+\frac{K^2}{2})h},
\end{align*}
which concludes the proof.

\section{Proof of Lemma \ref{kernel_u}}\label{kernel_u_proof}
We first define the following random variables
\begin{align*}
&x(h)=x(0)+\int_0^{h}b(x(s),u)ds+\int_0^h\sigma(x(s),u)dB(s)\\
&y(h)=x(0)+\int_0^{h}b(y(s),\hat{u})ds+\int_0^h\sigma(y(s),\hat{u})dB(s).
\end{align*}
By defining  $z(h):=x(h)-y(s)$, whose dynamics are given by
\begin{align*}
z(h)=\int_0^h \bar{b}(x(s),y(s),u,\hat{u})ds+\int_0^h \bar{\sigma}(x(s),y(s),u,\hat{u})dB(s)
\end{align*}
where 
\begin{align*}
\bar{b}(x(s),y(s),u,\hat{u})&:=b(x(s),u)-b(y(s),\hat{u})\\
\bar{\sigma}(x(s),y(s),u,\hat{u})&:=\sigma(x(s),u)-\sigma(y(s),\hat{u}).
\end{align*}
Using the Holder's inequality, the boundedness, and the continuity properties of $b,\sigma$, under the assumption that $h<1$, we can write
\begin{align*}
E\left[Z(h)^2\right]\leq& E\left[\int_0^h\left(KZ(s)+K|u-\hat{u}|\right)^2ds\right]+E\left[\left(\int_0^h\left(KZ(s)+K|u-\hat{u}|\right)dB(s)\right)^2\right]\\
&=2 E\left[\int_0^h\left(KZ(s)+K|u-\hat{u}|\right)^2ds\right]\leq 4K^2\int_0^hE\left[Z(s)^2\right]ds+4K^2 h|u-\hat{u}|^2
\end{align*}
under the assumption that $h<1$, and using the Gronwall inequality
\begin{align*}
E\left[Z(h)^2\right]\leq 4K^2 |u-\hat{u}|^2e^{4K^2h}.
\end{align*}
Using the Holder's inequality, we can conclude that
\begin{align*}
E\left[|X(h)-Y(h)|\right]\leq 2K e^{2K^2h}|u-\hat{u}|.
\end{align*}

\bibliographystyle{plain}

\bibliography{AliBibliography,references_acc,SerdarBibliography_acc,references}

\end{document}